\documentclass[11pt,reqno]{amsart}
\usepackage{amsmath}
\usepackage{amsfonts}
\usepackage{amssymb,latexsym}
\usepackage{cite}
\usepackage[height=190mm,width=130mm]{geometry}

\theoremstyle{plain}
\newtheorem{theorem}{Theorem}
\newtheorem{lemma}{Lemma}
\newtheorem{corollary}{Corollary}

\theoremstyle{definition}

\theoremstyle{remark}

\numberwithin{equation}{section}
\input{tcilatex}

\begin{document}
\title[SOME BASIC PROPERTIES OF THE GENERALIZED BI-PERIODIC FIBONACCI AND
LUCAS SEQUENCES]{SOME BASIC PROPERTIES OF THE GENERALIZED BI-PERIODIC
FIBONACCI AND LUCAS SEQUENCES}
\author{Elif TAN}
\address{Department of Mathematics, Ankara University, Science Faculty,
06100 Tandogan Ankara, Turkey.}
\email{etan@ankara.edu.tr}
\author{Ho-Hon Leung}
\address{Department of Mathematical Sciences, UAEU, Al-Ain, United Arab
Emirates}
\email{hohon.leung@uaeu.ac.ae}

\begin{abstract}
In this paper, we consider a generalization of Horadam sequence $\left\{
w_{n}\right\} $ which is defined by the recurrence relation $w_{n}=\chi
\left( n\right) w_{n-1}+cw_{n-2},$ where $\chi \left( n\right) =a$ if $n$ is
even, $\chi \left( n\right) =b$ if $n$ is odd with arbitrary initial
conditions $w_{0},w_{1}$ and nonzero real numbers $a,b$ and $c.$ As a
special case, by taking initial conditions $0,1$ and \ $2,b$ we define the
sequences $\left\{ u_{n}\right\} $ and $\left\{ v_{n}\right\} ,$
respectively. The main purpose of this study is to derive some basic
properties of the sequences $\left\{ u_{n}\right\} ,\left\{ v_{n}\right\} $
and $\left\{ w_{n}\right\} $ by using matrix approach.
\end{abstract}

\maketitle

\section{Introduction}

A generalization of Horadam sequence $\left\{ w_{n}\right\} $ is defined by
the recurrence relation%
\begin{equation}
w_{n}=\chi \left( n\right) w_{n-1}+cw_{n-2},\text{ }n\geq 2  \label{1}
\end{equation}%
where $\chi \left( n\right) =a$ if $n$ is even, $\chi \left( n\right) =b$ if 
$n$ is odd with arbitrary initial conditions $w_{0},w_{1}$ and nonzero real
numbers $a,b$ and $c.$ They emerged as a generalization of the best known
sequences in the literature, such as Horadam sequence, Fibonacci\&Lucas
sequence, $k$-Fibonacci\&$k$-Lucas sequence, Pell\&Pell-Lucas sequence,
Jacobsthal\&Jacobsthal-Lucas sequence, etc. Here we call the sequence $%
\left\{ w_{n}\right\} $ as \textit{generalized bi-periodic Horadam sequence}%
. In particular, by taking initial conditions $0,1$ and $2,b$ we call these
sequences as \textit{generalized bi-periodic Fibonacci sequence} $\{u_{n}\}$
and \textit{generalized bi-periodic Lucas sequence} $\{v_{n}\},$
respectively.

Some modified versions of the sequence $\left\{ w_{n}\right\} $ have been
studied by several authors. For the case of $w_{0}=0,w_{1}=1$ and $c=1,$ the
sequence $\left\{ w_{n}\right\} $ reduces to the bi-periodic Fibonacci
sequence, and some basic properties of this sequence can be found in \cite%
{Edson, Yayenie, Sahin}. Its companion sequence, bi-periodic Lucas sequence,
was studied in \cite{Irmak, Bilgici, Tan1, Tan5}. For the case of $c=1,$ the
sequence $\left\{ w_{n}\right\} $ reduces to the bi-periodic Horadam
sequence,\textit{\ }and several properties of this sequence were given in 
\cite{Edson, Tan3}. For a further generalization of the sequence $\left\{
w_{n}\right\} ,$ we refer to \cite{Irmak2, Panario}.

On the other hand, the matrix method is extremely useful for obtaining some
well-known Fibonacci properties, such as Cassini's identity, d'Ocagne's
identity, and convolution property, etc. For the detailed history of the
matrix technique see \cite{Kalman, Williams, Johnson, Waddill, siar,
Demirturk}. The $2\times 2$ matrix representation for the general case of
the sequence $\left\{ w_{n}\right\} $ was given firstly in \cite{Tan1}, and
several properties were obtained for the even indices terms of this
sequence. Then, in \cite{Tan2}, author defined a new matrix identity for the
bi-periodic Fibonacci sequence as follows:%
\begin{equation}
S:=\left( 
\begin{array}{cc}
ab & ab \\ 
1 & 0%
\end{array}%
\right) \Rightarrow S^{n}=\left( ab\right) ^{\left\lfloor \frac{n}{2}%
\right\rfloor }\left( 
\begin{array}{cc}
b^{\zeta \left( n\right) }q_{n+1} & a^{\zeta \left( n\right) }bq_{n} \\ 
a^{-\zeta \left( n+1\right) }q_{n} & b^{\zeta \left( n\right) }q_{n-1}%
\end{array}%
\right) ,  \label{2}
\end{equation}%
where $\left\{ q_{n}\right\} $ is the bi-periodic Fibonacci sequence and $%
\zeta \left( n\right) $ is the parity function. By using this matrix
identity, simple proofs of several identities of the bi-periodic Fibonacci
and Lucas numbers were given. One of the main objectives of this study is to
generalize the matrix identity (\ref{2}) for the sequence $\left\{
w_{n}\right\} $.

Similar to the notation of the classical Horadam sequence in \cite{Horadam},
we can state several number of sequences in terms of the generalized
bi-periodic Horadam sequence $\left\{ w_{n}\right\} :=$ $\{w_{n}\left(
w_{0},w_{1};a,b,c\right) \}$ in the following table.

$%
\begin{array}{c}
\\ 
\begin{tabular}{|l|l|l|}
\hline
$\{w_{n}\}$ & $\{w_{n}\left( w_{0},w_{1};a,b,c\right) \}$ & generalized
bi-periodic Horadam sequence \\ \hline\hline
$\{u_{n}\}$ & $\{w_{n}\left( 0,1;a,b,c\right) \}$ & generalized bi-periodic
Fibonacci sequence \\ \hline
$\{v_{n}\}$ & $\{w_{n}\left( 2,b;a,b,c\right) \}$ & generalized bi-periodic
Lucas sequence \\ \hline
$\{q_{n}\}$ & $\{w_{n}\left( 0,1;a,b,1\right) \}$ & bi-periodic Fibonacci
sequence \cite{Edson} \\ \hline
$\{p_{n}\}$ & $\{w_{n}\left( 2,a;b,a,1\right) \}$ & bi-periodic Lucas
sequence \cite{Bilgici} \\ \hline
$\{W_{n}\}$ & $\{w_{n}\left( w_{0},w_{1};a,b,1\right) \}$ & bi-periodic
Horadam sequence \cite{Edson} \\ \hline
$\{H_{n}\}$ & $\{w_{n}\left( w_{0},w_{1};p,p,-q\right) \}$ & Horadam
sequence \cite{Horadam} \\ \hline
$\{F_{n}\}$ & $\{w_{n}\left( 0,1;1,1,1\right) \}$ & Fibonacci sequence \\ 
\hline
$\{L_{n}\}$ & $\{w_{n}\left( 2,1;1,1,1\right) \}$ & Lucas sequence \\ \hline
$\{F_{k,n}\}$ & $\{w_{n}\left( 0,1;k,k,1\right) \}$ & $k$-Fibonacci sequence
\\ \hline
$\{L_{k,n}\}$ & $\{w_{n}\left( 0,k;k,k,1\right) \}$ & $k$-Lucas sequence \\ 
\hline
$\{P_{n}\}$ & $\{w_{n}\left( 0,1;2,2,1\right) \}$ & Pell sequence \\ \hline
$\{PL_{n}\}$ & $\{w_{n}\left( 2,2;2,2,1\right) \}$ & Pell-Lucas sequence \\ 
\hline
$\{J_{n}\}$ & $\{w_{n}\left( 0,1;1,1,2\right) \}$ & Jacobsthal sequence \\ 
\hline
$\{JL_{n}\}$ & $\{w_{n}\left( 2,1;1,1,2\right) \}$ & Jacobsthal-Lucas
sequence \\ \hline
\end{tabular}
\\ 
\\ 
Table\text{ }1\label{Table1}:\text{Special cases of the sequence }\{w_{n}\}
\\ 
\\ 
\\ 
\end{array}%
$

The outline of this paper as follows: In Section 2, inspired by the matrix
identity (\ref{2}), we give analogous matrix representations for the
generalized bi-periodic Fibonacci and the generalized bi-periodic Lucas
numbers. Then, we generalize the matrix identity (\ref{2}) to the
generalized bi-periodic Horadam numbers. Thus, one can develop many matrix
identities by choosing appropriate initial values in our matrix formula. We
state several properties of these numbers by using matrix approach which
provides a very simple proof. Section 3 is devoted to obtain more
generalized expressions for the generalized bi-periodic Horadam numbers, by
using the matrix method in \cite{Waddill}. \vspace{1.5cc}

\begin{center}
\textbf{2. MATRIX REPRESENTATIONS FOR $\{u_{n}\}$, $\{v_{n}\}$ and $%
\{w_{n}\} $}
\end{center}

First, we define the matrix $U:=\left( 
\begin{array}{cc}
ab & cb \\ 
a & 0%
\end{array}%
\right) .$ For any nonnegative integer $n,$ by using induction, we have%
\begin{equation}
U^{n}=\left( ab\right) ^{\left\lfloor \frac{n}{2}\right\rfloor }\left( 
\begin{array}{cc}
b^{\zeta \left( n\right) }u_{n+1} & cba^{-\zeta \left( n+1\right) }u_{n} \\ 
a^{\zeta \left( n\right) }u_{n} & cb^{\zeta \left( n\right) }u_{n-1}%
\end{array}%
\right) ,  \label{7}
\end{equation}%
where $u_{n}$ is the $n-$th generalized bi-periodic Fibonacci number. Since
the matrix $U$ is invertible, then%
\begin{equation*}
U^{-n}=\frac{\left( ab\right) ^{\left\lfloor \frac{n}{2}\right\rfloor }}{%
\left( -abc\right) ^{n}}\left( 
\begin{array}{cc}
cb^{\zeta \left( n\right) }u_{n-1} & -cba^{-\zeta \left( n+1\right) }u_{n}
\\ 
-a^{\zeta \left( n\right) }u_{n} & b^{\zeta \left( n\right) }u_{n+1}%
\end{array}%
\right) .
\end{equation*}

By using the matrix identity (\ref{7}) and using the similar method in \cite[%
Theorem 1]{Tan2}, one can obtain the following results which give some basic
properties of $\left\{ u_{n}\right\} $. Note that the results $\left(
1\right) -\left( 3\right) $ can be found in \cite[Theorem 9]{Yayenie}, but
here we obtain these identities by using matrix approach.

\begin{lemma}
\label{l1} The sequence $\{u_{n}\}$ satisfies the following identities:

\begin{enumerate}
\item $%
\begin{array}{c}
\left( \frac{a}{b}\right) ^{\zeta \left( n\right) }u_{n}^{2}-\left( \frac{a}{%
b}\right) ^{\zeta \left( n+1\right) }u_{n-1}u_{n+1}=\frac{a}{b}\left(
-c\right) ^{n-1}%
\end{array}%
$

\item $%
\begin{array}{c}
\left( \frac{b}{a}\right) ^{\zeta \left( mn+n\right) }u_{m}u_{n+1}+\left( 
\frac{b}{a}\right) ^{\zeta \left( mn+m\right) }cu_{n}u_{m-1}=u_{n+m}%
\end{array}%
$

\item $%
\begin{array}{c}
\left( \frac{b}{a}\right) ^{\zeta \left( mn+n\right) }u_{n}u_{m+1}-\left( 
\frac{b}{a}\right) ^{\zeta \left( mn+m\right) }u_{m}u_{n+1}=\left( -c\right)
^{m}u_{n-m}%
\end{array}%
$

\item $%
\begin{array}{c}
\left( \frac{b}{a}\right) ^{\zeta \left( mn+n\right) }u_{m}u_{n-m+1}+c\left( 
\frac{b}{a}\right) ^{\zeta \left( mn\right) }u_{m-1}u_{n-m}=u_{n}%
\end{array}%
$
\end{enumerate}
\end{lemma}

Now we consider the matrix equality%
\begin{equation}
K:=\frac{1}{2}\left( 
\begin{array}{cc}
ab & \Delta \\ 
1 & ab%
\end{array}%
\right) \Rightarrow K^{n}=\frac{\left( ab\right) ^{\left\lfloor \frac{n}{2}%
\right\rfloor }}{2}\left( 
\begin{array}{cc}
a^{\zeta \left( n\right) }v_{n} & \Delta a^{\zeta \left( n\right) -1}u_{n}
\\ 
a^{\zeta \left( n\right) -1}u_{n} & a^{\zeta \left( n\right) }v_{n}%
\end{array}%
\right) ,  \label{18}
\end{equation}%
where $\Delta :=a^{2}b^{2}+4abc\neq 0.$ By using the method in \cite[Theorem
4]{Tan2}, one can obtain the following results which give some relations
involving both the generalized bi-periodic Fibonacci and the generalized
bi-periodic Lucas numbers.

\begin{lemma}
\label{l}The sequences $\{u_{n}\}$ and $\{v_{n}\}$ satisfy the following
identities:

\begin{enumerate}
\item $%
\begin{array}{c}
v_{n}^{2}-\frac{\Delta }{a^{2}}u_{n}^{2}=4\left( \frac{b}{a}\right) ^{\zeta
\left( n\right) }\left( -c\right) ^{n}%
\end{array}%
$

\item $%
\begin{array}{c}
v_{m}v_{n}+\frac{\Delta }{a^{2}}u_{m}u_{n}=2\left( \frac{b}{a}\right)
^{\zeta \left( n\right) \zeta \left( m\right) }v_{n+m}%
\end{array}%
$

\item $%
\begin{array}{c}
u_{m}v_{n}+u_{n}v_{m}=2\left( \frac{b}{a}\right) ^{\zeta \left( n\right)
\zeta \left( m\right) }u_{n+m}%
\end{array}%
$

\item $%
\begin{array}{c}
v_{m}v_{n}-\frac{\Delta }{a^{2}}u_{m}u_{n}=2\left( -c\right) ^{m}\left( 
\frac{a}{b}\right) ^{-\zeta \left( n\right) \zeta \left( m\right) }v_{n-m}%
\end{array}%
$

\item $%
\begin{array}{c}
u_{n}v_{m}-u_{m}v_{n}=2\left( -c\right) ^{m}\left( \frac{a}{b}\right)
^{-\zeta \left( n\right) \zeta \left( m\right) }u_{n-m}%
\end{array}%
$

\item $%
\begin{array}{c}
v_{n+m}+\left( -c\right) ^{m}v_{n-m}=\left( \frac{a}{b}\right) ^{\zeta
\left( n\right) \zeta \left( m\right) }v_{m}v_{n}%
\end{array}%
$

\item $%
\begin{array}{c}
u_{n+m}+\left( -c\right) ^{m}u_{n-m}=\left( \frac{a}{b}\right) ^{\zeta
\left( n\right) \zeta \left( m\right) }u_{n}v_{m}.%
\end{array}%
$
\end{enumerate}
\end{lemma}

We define the matrix $H:=K+abcK^{-1}=\left( 
\begin{array}{cc}
0 & \Delta \\ 
1 & 0%
\end{array}%
\right) .$ It is clear that we have the matrix relation%
\begin{equation}
K^{n}=\frac{\left( ab\right) ^{\left\lfloor \frac{n}{2}\right\rfloor }}{2}%
\left( a^{\zeta \left( n\right) -1}u_{n}H+a^{\zeta \left( n\right)
}v_{n}I\right) .  \label{19}
\end{equation}%
Also from the relation%
\begin{equation}
w_{n}=u_{n}w_{1}+c\left( \frac{b}{a}\right) ^{\zeta \left( n\right)
}u_{n-1}w_{0},  \label{a}
\end{equation}%
%
%
%
we have $v_{n}=bu_{n}+2c\left( \frac{b}{a}\right) ^{\zeta \left( n\right)
}u_{n-1}.$ Then we have%
\begin{equation}
K^{n}=\left( ab\right) ^{\left\lfloor \frac{n}{2}\right\rfloor }\left(
a^{\zeta \left( n\right) -1}u_{n}K+cb^{\zeta \left( n\right)
}u_{n-1}I\right) .  \label{20}
\end{equation}%
By using the matrix relations (\ref{19}) and (\ref{20}), we give Theorem \ref%
{t2} and Theorem \ref{t3}, repectively.

\begin{theorem}
\label{t2}Let $D:=1-a^{\zeta \left( m\right) }v_{m}+\left( ab\right) ^{\zeta
\left( m\right) }\left( -c\right) ^{m}\neq 0,$ then%
\begin{eqnarray*}
&&\sum_{j=0}^{n}\left( ab\right) ^{\left\lfloor \frac{mj+r}{2}\right\rfloor
}a^{\zeta \left( mj+r\right) -1}u_{mj+r} \\
&=&\frac{1}{D}\left( \left( ab\right) ^{\left\lfloor \frac{r}{2}%
\right\rfloor }a^{\zeta \left( r\right) -1}\left( u_{r}-\left( -c\right)
^{m}a^{\zeta \left( m\right) \zeta \left( r+1\right) }b^{\zeta \left(
m\right) \zeta \left( r\right) }u_{r-m}\right) \right. \\
&&\left. -\left( ab\right) ^{\left\lfloor \frac{mn+m+r}{2}\right\rfloor
}a^{\zeta \left( mn+m+r\right) -1}\right. \\
&&\left. \times \left( u_{mn+m+r}+\left( -c\right) ^{m}a^{\zeta \left(
m\right) \zeta \left( mn+m+r+1\right) }b^{\zeta \left( m\right) \zeta \left(
mn+m+r\right) }u_{mn+r}\right) \right) ,
\end{eqnarray*}%
\begin{eqnarray*}
&&\sum_{j=0}^{n}\left( ab\right) ^{\left\lfloor \frac{mj+r}{2}\right\rfloor
}a^{\zeta \left( mj+r\right) }v_{mj+r} \\
&=&\frac{1}{D}\left( \left( ab\right) ^{\left\lfloor \frac{r}{2}%
\right\rfloor }a^{\zeta \left( r\right) }\left( v_{r}-\left( -c\right)
^{m}a^{\zeta \left( m\right) \zeta \left( r+1\right) }b^{\zeta \left(
m\right) \zeta \left( r\right) }v_{r-m}\right) \right. \\
&&\left. -\left( ab\right) ^{\left\lfloor \frac{mn+m+r}{2}\right\rfloor
}a^{\zeta \left( mn+m+r\right) }\right. \\
&&\left. \times \left( v_{mn+m+r}+\left( -c\right) ^{m}a^{\zeta \left(
m\right) \zeta \left( mn+m+r+1\right) }b^{\zeta \left( m\right) \zeta \left(
mn+m+r\right) }v_{mn+r}\right) \right) .
\end{eqnarray*}
\end{theorem}

\begin{proof}
It is clear that%
\begin{equation*}
I-\left( K^{m}\right) ^{n+1}=\left( I-K^{m}\right) \sum_{j=0}^{n}K^{mj}.
\end{equation*}%
Since 
\begin{equation*}
\det \left( I-K^{m}\right) =1-a^{\zeta \left( m\right) }v_{m}+\left(
ab\right) ^{\zeta \left( m\right) }\left( -c\right) ^{m}0\neq 0,
\end{equation*}%
then%
\begin{eqnarray*}
\left( I-K^{m}\right) ^{-1} &=&\frac{1}{D}\left( 
\begin{array}{cc}
1-a^{\zeta \left( m\right) }\frac{v_{m}}{2} & \Delta a^{\zeta \left(
m\right) -1}\frac{u_{m}}{2} \\ 
a^{\zeta \left( m\right) -1}\frac{u_{m}}{2} & 1-a^{\zeta \left( m\right) }%
\frac{v_{m}}{2}%
\end{array}%
\right) \\
&=&\frac{1}{D}\left( \left( 1-a^{\zeta \left( m\right) }\frac{v_{m}}{2}%
\right) I+a^{\zeta \left( m\right) -1}\frac{u_{m}}{2}H\right) .
\end{eqnarray*}%
By using the matrix identity (\ref{18}), we have%
\begin{equation*}
\left( I-K^{m}\right) ^{-1}\left( I-\left( K^{m}\right) ^{n+1}\right) K^{r}%
\text{\ \ \ \ \ \ \ \ \ \ \ \ \ \ \ \ \ \ \ \ \ \ \ \ \ \ \ \ \ \ \ \ \ \ \
\ \ \ \ \ \ \ \ \ \ \ \ \ \ \ \ \ \ \ \ \ \ \ \ \ \ \ \ \ \ \ \ \ \ \ \ \ \
\ \ \ }
\end{equation*}%
\begin{equation*}
=\sum_{j=0}^{n}K^{mj+r}\text{ \ \ \ \ \ \ \ \ \ \ \ \ \ \ \ \ \ \ \ \ \ \ \
\ \ \ \ \ \ \ \ \ \ \ \ \ \ \ \ \ \ \ \ \ \ \ \ \ \ \ \ \ \ \ \ \ \ \ \ \ \
\ \ \ \ \ \ \ \ \ \ \ \ \ \ \ \ \ \ \ \ \ \ \ \ \ \ \ \ \ \ \ \ \ \ \ \ \ }
\end{equation*}%
\begin{equation}
=\left( 
\begin{array}{cc}
\sum_{j=0}^{n}\left( ab\right) ^{\left\lfloor \frac{mj+r}{2}\right\rfloor
}a^{\zeta \left( mj+r\right) }\frac{v_{mj+r}}{2} & \Delta
\sum_{j=0}^{n}\left( ab\right) ^{\left\lfloor \frac{mj+r}{2}\right\rfloor
}a^{\zeta \left( mj+r\right) -1}\frac{u_{mj+r}}{2} \\ 
\sum_{j=0}^{n}\left( ab\right) ^{\left\lfloor \frac{mj+r}{2}\right\rfloor
}a^{\zeta \left( mj+r\right) -1}\frac{u_{mj+r}}{2} & \sum_{j=0}^{n}\left(
ab\right) ^{\left\lfloor \frac{mj+r}{2}\right\rfloor }a^{\zeta \left(
mj+r\right) }\frac{v_{mj+r}}{2}%
\end{array}%
\right) .  \label{*}
\end{equation}%
On the other hand, we have%
\begin{equation*}
\left( I-K^{m}\right) ^{-1}\left( K^{r}-K^{mn+m+r}\right) \text{ \ \ \ \ \ \
\ \ \ \ \ \ \ \ \ \ \ \ \ \ \ \ \ \ \ \ \ \ \ \ \ \ \ \ \ \ \ \ \ \ \ \ \ \
\ \ \ \ \ \ \ \ \ \ \ \ \ \ \ \ \ \ \ \ \ \ \ \ \ \ \ \ \ \ \ \ \ \ \ \ \ \
\ \ \ \ \ \ \ \ \ }
\end{equation*}%
\begin{equation*}
=\frac{1}{D}\left( \left( 1-a^{\zeta \left( m\right) }\frac{v_{m}}{2}\right)
I+a^{\zeta \left( m\right) -1}\frac{u_{m}}{2}H\right) \left(
K^{r}-K^{mn+m+r}\right) \text{ \ \ \ \ \ \ \ \ \ \ \ \ \ \ \ \ \ \ \ \ \ }
\end{equation*}%
\begin{equation*}
=\frac{1}{D}\left( \left( 1-a^{\zeta \left( m\right) }\frac{v_{m}}{2}\right)
\left( K^{r}-K^{mn+m+r}\right) +a^{\zeta \left( m\right) -1}\frac{u_{m}}{2}%
H\left( K^{r}-K^{mn+m+r}\right) \right)
\end{equation*}%
\begin{equation}
=\frac{1}{D}\left( \left( 1-a^{\zeta \left( m\right) }\frac{v_{m}}{2}\right)
\left( 
\begin{array}{cc}
X & \Delta Y \\ 
Y & X%
\end{array}%
\right) +a^{\zeta \left( m\right) -1}\frac{u_{m}}{2}\left( 
\begin{array}{cc}
\Delta Y & \Delta X \\ 
X & \Delta Y%
\end{array}%
\right) \right) \text{ \ \ \ \ \ \ \ \ \ \ }  \label{***}
\end{equation}%
where%
\begin{equation*}
X:=\left( ab\right) ^{\left\lfloor \frac{r}{2}\right\rfloor }a^{\zeta \left(
r\right) }\frac{v_{r}}{2}-\left( ab\right) ^{\left\lfloor \frac{mn+m+r}{2}%
\right\rfloor }a^{\zeta \left( mn+m+r\right) }\frac{v_{mn+m+r}}{2},\text{ \
\ \ \ \ }
\end{equation*}%
\begin{equation*}
Y:=\left( ab\right) ^{\left\lfloor \frac{r}{2}\right\rfloor }a^{\zeta \left(
r\right) -1}\frac{u_{r}}{2}-\left( ab\right) ^{\left\lfloor \frac{mn+m+r}{2}%
\right\rfloor }a^{\zeta \left( mn+m+r\right) -1}\frac{u_{mn+m+r}}{2}.
\end{equation*}
By equating the corresponding entries of (\ref{*}) and (\ref{***}), and
using the identity \textit{5.} of Lemma \ref{l}, we get the desired result.
The remaining result can be proven similarly by using the identity \textit{4.%
} of Lemma \ref{l}.
\end{proof}

\begin{theorem}
\label{t3}For any nonnegative integers $n,r$ and $m$ with $m>1$, we have%
\begin{equation*}
u_{mn+r}=\frac{a^{1-\zeta \left( mn+r\right) }}{\left( ab\right)
^{\left\lfloor \frac{mn+r}{2}\right\rfloor }}\sum_{i=0}^{n}\dbinom{n}{i}%
c^{n-i}u_{m}^{i}u_{m-1}^{n-i}u_{i+r}\delta \lbrack m,n,r,i],
\end{equation*}%
\begin{equation*}
v_{mn+r}=\frac{a^{1-\zeta \left( mn+r\right) }}{\left( ab\right)
^{\left\lfloor \frac{mn+r}{2}\right\rfloor }}\sum_{i=0}^{n}\dbinom{n}{i}%
c^{n-i}u_{m}^{i}u_{m-1}^{n-i}v_{i+r}\delta \lbrack m,n,r,i]
\end{equation*}%
where%
\begin{equation*}
\delta \lbrack m,n,r,i]:=\left( ab\right) ^{\left\lfloor \frac{i+r}{2}%
\right\rfloor +n\left\lfloor \frac{m}{2}\right\rfloor }a^{-\zeta \left(
m+1\right) i-1+\zeta \left( i+r\right) }b^{\zeta \left( m\right) \left(
n-i\right) }.
\end{equation*}
\end{theorem}

\begin{proof}
By considering the matrix identities (\ref{20}) and (\ref{18}), then
equating the corresponding entries we obtain the desired results.
\end{proof}

Note that Theorem \ref{t2} and Theorem \ref{t3} can be seen as a
generalization of the results in \cite{siar}.

Finally, we define the matrix $T:=\left( 
\begin{array}{cc}
abw_{1}+cbw_{0} & cbw_{1} \\ 
aw_{1} & cbw_{0}%
\end{array}%
\right) .$ By induction we have%
\begin{equation}
TU^{n}=\left( ab\right) ^{\left\lfloor \frac{n+1}{2}\right\rfloor }\left( 
\begin{array}{cc}
b^{\zeta \left( n+1\right) }w_{n+2} & cba^{-\zeta \left( n\right) }w_{n+1}
\\ 
a^{\zeta \left( n+1\right) }w_{n+1} & cb^{\zeta \left( n+1\right) }w_{n}%
\end{array}%
\right) ,  \label{8}
\end{equation}%
where $w_{n}$ is the $n-$th generalized bi-periodic Horadam number.

If we take the determinant of both sides of the equation (\ref{8}) and
taking $n\rightarrow n-1$, we obtain the Cassini's identity for the sequence 
$\left\{ w_{n}\right\} $ as%
\begin{equation}
\left( \frac{b}{a}\right) ^{\zeta \left( n\right) }w_{n-1}w_{n+1}-\left( 
\frac{b}{a}\right) ^{\zeta \left( n+1\right) }w_{n}^{2}=\left( -1\right)
^{n}c^{n-1}\left( w_{1}^{2}-bw_{0}w_{1}-c\frac{b}{a}w_{0}^{2}\right) .
\label{c}
\end{equation}

The matrix $T$ can be written as%
\begin{equation}
T=cbw_{0}I+w_{1}U,  \label{9}
\end{equation}%
where $I$ is the $2\times 2$ unit matrix. It is easy to see that%
\begin{equation}
TU^{n}=cbw_{0}U^{n}+w_{1}U^{n+1}.  \label{10}
\end{equation}%
If we equate the corresponding entries of the matrix equality (\ref{10}), we
get the identity (\ref{a}). Also, the generalized bi-periodic Horadam
numbers for negative subscripts can be defined as%
\begin{equation}
w_{-n}=-\frac{a^{\zeta \left( n+1\right) }b^{\zeta \left( n\right) }}{c}%
w_{-n+1}+\frac{1}{c}w_{-n+2},  \label{11}
\end{equation}%
so that the matrix identity (\ref{8}) holds for every integer $n$. From (\ref%
{10}), we have%
\begin{equation}
\left( -c\right) ^{n}w_{-n}=\left( \frac{b}{a}\right) ^{\zeta \left(
n\right) }w_{0}u_{n+1}-w_{1}u_{n},  \label{12}
\end{equation}%
which reduce to%
\begin{equation}
\begin{array}{ccc}
u_{-n}=\frac{\left( -1\right) ^{n+1}}{c^{n}}u_{n} & \text{and} & v_{-n}=%
\frac{\left( -1\right) ^{n}}{c^{n}}v_{n}%
\end{array}
\label{6}
\end{equation}%
for the generalized bi-periodic Fibonacci and Lucas numbers, respectively. 
\vspace{1.5cc}

\begin{center}
\textbf{3. MORE GENERAL RESULTS FOR $\{w_{n}\}$}
\end{center}

Besides the matrix $U$, the $n^{\text{th}}$ power of the matrix $A:=\left( 
\begin{array}{cc}
ab & abc \\ 
1 & 0%
\end{array}%
\right)$ also has entries involving generalized bi-periodic Fibonacci
numbers; that is, 
\begin{align}  \label{30}
A^n=\left( 
\begin{array}{cc}
ab & abc \\ 
1 & 0%
\end{array}%
\right)^n =\left( ab\right) ^{\left\lfloor \frac{n}{2}\right\rfloor }\left( 
\begin{array}{cc}
b^{\zeta \left( n\right) }u_{n+1} & cba^{\zeta \left( n\right) }u_{n} \\ 
a^{-\zeta \left( n+1\right) }u_{n} & cb^{\zeta \left( n\right) }u_{n-1}%
\end{array}%
\right).
\end{align}%
If $n$ is even, then 
\begin{align}  \label{31}
A^n \left(%
\begin{array}{c}
w_1 \\ 
a^{-1} w_0%
\end{array}
\right)=(ab)^{\frac{n}{2}} \left( 
\begin{array}{c}
w_{n+1} \\ 
a^{-1} w_n%
\end{array}
\right), \quad A^n \left( 
\begin{array}{c}
cb w_2 \\ 
c w_1%
\end{array}%
\right) =(ab)^{\frac{n}{2}} \left(%
\begin{array}{c}
cb w_{n+2} \\ 
c w_{n+1}%
\end{array}
\right)
\end{align}%
By combining the equations (\ref{30}) and (\ref{31}) for even $n$, 
\begin{align}  \label{32}
\begin{pmatrix}
w_{n+1} \\ 
a^{-1} w_n%
\end{pmatrix}%
&= 
\begin{pmatrix}
u_{n+1} & cb u_{n} \\ 
a^{-1 }u_{n} & c u_{n-1}%
\end{pmatrix}%
\begin{pmatrix}
w_1 \\ 
a^{-1} w_0%
\end{pmatrix}%
, 
\begin{pmatrix}
cbw_{n+2} \\ 
c w_{n+1}%
\end{pmatrix}
& = 
\begin{pmatrix}
u_{n+1} & cb u_{n} \\ 
a^{-1 }u_{n} & c u_{n-1}%
\end{pmatrix}%
\begin{pmatrix}
cb w_2 \\ 
c w_1%
\end{pmatrix}%
.
\end{align}%
We get (\ref{a}) by comparing entries in (\ref{32}). We generalize it
further as follows.

\begin{theorem}
\label{t4} Let $n$ and $p$ be any positive integers. Then 
\begin{align}  \label{100}
w_{n+p} &= \left(\frac{b}{a} \right)^{\zeta (n+1)\zeta (p)} u_n w_{p+1} +c
\left( \frac{b}{a} \right)^{\zeta (n)\zeta (p+1)} u_{n-1} w_p.
\end{align}
\end{theorem}

\begin{proof}
Let $n$ and $p$ be even. By (\ref{31}), (\ref{32}), 
\begin{align}
(ab)^{\frac{n+p}{2}}\left( 
\begin{array}{c}
w_{n+p+1} \\ 
a^{-1}w_{n+p}%
\end{array}%
\right) & =A^{n+p}\left( 
\begin{array}{c}
w_{1} \\ 
a^{-1}w_{0}%
\end{array}%
\right) =(ab)^{\frac{p}{2}}A^{n}\left( 
\begin{array}{c}
w_{p+1} \\ 
a^{-1}w_{p}%
\end{array}%
\right)  \label{34} \\
& =(ab)^{\frac{n+p}{2}}\left( 
\begin{array}{cc}
u_{n+1} & cbu_{n} \\ 
a^{-1}u_{n} & cu_{n-1}%
\end{array}%
\right) \left( 
\begin{array}{c}
w_{p+1} \\ 
a^{-1}w_{p}%
\end{array}%
\right) .  \notag
\end{align}%
By comparing both entries of the matrices on both sides of (\ref{34}), we
get (\ref{100}). Similarly, we obtain the following equation by (\ref{31})
and (\ref{32}), 
\begin{equation}
\left( 
\begin{array}{c}
cbw_{n+p+2} \\ 
cw_{n+p+1}%
\end{array}%
\right) =\left( 
\begin{array}{cc}
u_{n+1} & cbu_{n} \\ 
a^{-1}u_{n} & cu_{n-1}%
\end{array}%
\right) \left( 
\begin{array}{c}
cbw_{p+2} \\ 
cw_{p+1}%
\end{array}%
\right) .  \label{35}
\end{equation}%
By comparing entries of the matrices in (\ref{35}), we get the desired
result.
\end{proof}

\noindent By Theorem \ref{t4}, we have the following matrix identities for
even $n$ and even $p$: 
\begin{align}
\left( 
\begin{array}{c}
w_{n+p+1} \\ 
a^{-1}w_{p}%
\end{array}%
\right) & =\left( 
\begin{array}{cc}
u_{n+1} & bcu_{n} \\ 
0 & 1%
\end{array}%
\right) \left( 
\begin{array}{c}
w_{p+1} \\ 
a^{-1}w_{p}%
\end{array}%
\right) ,  \label{36} \\
\left( 
\begin{array}{c}
cbw_{n+p+2} \\ 
cw_{p+1}%
\end{array}%
\right) & =\left( 
\begin{array}{cc}
u_{n+1} & bcu_{n} \\ 
0 & 1%
\end{array}%
\right) \left( 
\begin{array}{c}
cbw_{p+2} \\ 
cw_{p+1}%
\end{array}%
\right) .  \label{37}
\end{align}

\noindent The following theorem is a generalization of Catalan's identity,
Cassini's identity and d'Ocagne's identity.

\begin{theorem}
\label{t5} Let $n$, $p$ and $q$ be any positive integers, then we have the
following identity: 
\begin{eqnarray*}
&&\left( \frac{b}{a}\right)^{\zeta(n)\zeta(p)\zeta(q)} w_{n+p} w_{n+q} -
\left(\frac{b}{a} \right)^{\zeta(n+1)\zeta(p)\zeta(q)} w_n w_{n+p+q} \\
&=&\left( \frac{b}{a}\right)^{\zeta(n)\zeta(p+1)\zeta(q+1)}(-c)^n u_p u_q
\left( w_1^2 - b w_0 w_1 -\frac{b}{a} c w_0^2 \right).
\end{eqnarray*}
\end{theorem}

\begin{proof}
For the case of even $n$, odd $p$ and odd $q$, we note that 
\begin{equation}
c\left( w_{n+p}w_{n+q}-\frac{b}{a}w_{n}w_{n+p+q}\right) =\left( 
\begin{array}{cc}
w_{n+q} & a^{-1}w_{n}%
\end{array}%
\right) \left( 
\begin{array}{c}
cw_{n+p} \\ 
-bcw_{n+p+q}%
\end{array}%
\right) .  \label{37.5}
\end{equation}%
By (\ref{36}) and (\ref{37}) , we get the following two equations, 
\begin{align}
(ab)^{\frac{n}{2}}\left( 
\begin{array}{cc}
w_{n+q} & a^{-1}w_{n}%
\end{array}%
\right) & =(ab)^{\frac{n}{2}}\left( 
\begin{array}{cc}
w_{n+1} & a^{-1}w_{n}%
\end{array}%
\right) \left( 
\begin{array}{cc}
u_{q} & 0 \\ 
bcu_{q-1} & 1%
\end{array}%
\right)  \label{38} \\
& =\left( 
\begin{array}{cc}
w_{1} & a^{-1}w_{0}%
\end{array}%
\right) (A^{T})^{n}\left( 
\begin{array}{cc}
u_{q} & 0 \\ 
bcu_{q-1} & 1%
\end{array}%
\right) .  \notag
\end{align}%
\begin{equation}
(ab)^{\frac{n+p-1}{2}}\left( 
\begin{array}{c}
cw_{n+p} \\ 
-bcw_{n+p+q}%
\end{array}%
\right) =\left( 
\begin{array}{cc}
1 & 0 \\ 
-bcu_{q-1} & u_{q}%
\end{array}%
\right) \left( 
\begin{array}{cc}
0 & -1 \\ 
-abc & ab%
\end{array}%
\right) ^{n+p-1}\left( 
\begin{array}{c}
cw_{1} \\ 
-bcw_{2}%
\end{array}%
\right) .  \label{39}
\end{equation}%
We note that 
\begin{equation}
\left( 
\begin{array}{cc}
u_{q} & 0 \\ 
bcu_{q-1} & 1%
\end{array}%
\right) \left( 
\begin{array}{cc}
1 & 0 \\ 
-bcu_{q-1} & u_{q}%
\end{array}%
\right) =u_{q}I,\left( 
\begin{array}{cc}
ab & 1 \\ 
abc & 0%
\end{array}%
\right) \left( 
\begin{array}{cc}
0 & -1 \\ 
-abc & ab%
\end{array}%
\right) =-abcI.  \label{40}
\end{equation}%
We take the product of equations (\ref{38}) and (\ref{39}), and by using (%
\ref{37.5}), (\ref{40}), we get 
\begin{align}
& (ab)^{n+\frac{p-1}{2}}c\left( w_{n+p}w_{n+q}-\frac{b}{a}%
w_{n}w_{n+p+q}\right)  \notag \\
& =u_{q}\left( 
\begin{array}{cc}
w_{1} & a^{-1}w_{0}%
\end{array}%
\right) \left( 
\begin{array}{cc}
ab & 1 \\ 
abc & 0%
\end{array}%
\right) ^{n}\left( 
\begin{array}{cc}
0 & -1 \\ 
-abc & ab%
\end{array}%
\right) ^{n+p-1}\left( 
\begin{array}{c}
cw_{1} \\ 
-bcw_{2}%
\end{array}%
\right)  \notag \\
& =(-abc)^{n}u_{q}\left( 
\begin{array}{cc}
w_{1} & a^{-1}w_{0}%
\end{array}%
\right) \left( 
\begin{array}{cc}
0 & -1 \\ 
-abc & ab%
\end{array}%
\right) ^{p-1}\left( 
\begin{array}{c}
cw_{1} \\ 
-bcw_{2}%
\end{array}%
\right)  \notag \\
& =(-abc)^{n}u_{q}\left( 
\begin{array}{cc}
w_{1} & a^{-1}w_{0}%
\end{array}%
\right) (ab)^{\frac{p-1}{2}}\left( 
\begin{array}{cc}
cu_{p-2} & -a^{-1}u_{p-1} \\ 
-bcu_{p-1} & u_{p}%
\end{array}%
\right) \left( 
\begin{array}{c}
cw_{1} \\ 
-bcw_{2}%
\end{array}%
\right) .  \label{41}
\end{align}%
We compute the following matrix product: 
\begin{align}
& \left( 
\begin{array}{cc}
w_{1} & a^{-1}w_{0}%
\end{array}%
\right) \left( 
\begin{array}{cc}
cu_{p-2} & -a^{-1}u_{p-1} \\ 
-bcu_{p-1} & u_{p}%
\end{array}%
\right) \left( 
\begin{array}{c}
cw_{1} \\ 
-bcw_{2}%
\end{array}%
\right)  \notag \\
& =c^{2}w_{1}^{2}u_{p-2}-\frac{b}{a}c^{2}w_{0}w_{1}u_{p-1}+\frac{b}{a}%
cw_{1}w_{2}u_{p-1}-\frac{b}{a}cw_{0}w_{2}u_{p}  \notag \\
& =c^{2}w_{1}^{2}u_{p-2}-\frac{b}{a}cw_{1}u_{p-1}\left( cw_{0}-w_{2}\right) -%
\frac{b}{a}cw_{0}u_{p}(aw_{1}+cw_{0})  \notag \\
& =c^{2}w_{1}^{2}u_{p-2}+bcw_{1}^{2}u_{p-1}-bcw_{0}w_{1}u_{p}-\frac{b}{a}%
c^{2}w_{0}^{2}u_{p}  \notag \\
& =cw_{1}^{2}(u_{p-2}+bu_{p-1})-bcw_{0}w_{1}u_{p}-\frac{b}{a}%
c^{2}w_{0}^{2}u_{p}  \notag \\
& =cw_{1}^{2}u_{p}-bcw_{0}w_{1}u_{p}-\frac{b}{a}c^{2}w_{0}^{2}u_{p}=cu_{p}%
\left( w_{1}^{2}-bw_{0}w_{1}-\frac{b}{a}cw_{0}^{2}\right) .  \notag
\end{align}%
We replace it in (\ref{41}) to get the result as desired for the case of
even $n$, odd $p$ and odd $q$.

For the case of odd $n^{\prime }$, we take $n^{\prime }=n+1$ and then do a
similar computation as above for 
\begin{equation*}
(ab)^{n+\frac{p+1}{2}} \left(%
\begin{array}{cc}
bc w_{(n+1)+p} & c w_{(n+1)}%
\end{array}
\right)\left(%
\begin{array}{c}
a^{-1} w_{(n+1)+q} \\ 
-w_{(n+1)+p+q}%
\end{array}
\right)
\end{equation*}%
where $n$ is even, $p$, $q$ are odd.

For the case of even $n$, even $p$ and odd $q$, we do a similar computation
as above for 
\begin{equation*}
(ab)^{n+\frac{p}{2}} \left(%
\begin{array}{cc}
w_{n+q} & a^{-1} w_n%
\end{array}
\right) \left(%
\begin{array}{c}
a^{-1} w_{n+p} \\ 
- w_{n+p+q}%
\end{array}
\right).
\end{equation*}

For the case of odd $n$, even $p$ and odd $q$, we do a similar computation
as above for 
\begin{equation*}
(ab)^{n-1+\frac{p}{2} }\left(%
\begin{array}{cc}
bc w_{n+q} & c w_n%
\end{array}
\right) \left(%
\begin{array}{c}
c w_{n+p} \\ 
-bc w_{n+p+q}%
\end{array}
\right).
\end{equation*}

The other cases for even $q$ can be proven similarly.
\end{proof}

We state another two matrix identities for even $n$: 
\begin{equation}
\begin{pmatrix}
w_{1} & bcw_{0}%
\end{pmatrix}%
A^{n}=(ab)^{\frac{n}{2}}%
\begin{pmatrix}
w_{n+1} & bcw_{n}%
\end{pmatrix}%
,%
\begin{pmatrix}
a^{-1}w_{2} & cw_{1}%
\end{pmatrix}%
A^{n}=(ab)^{\frac{n}{2}}%
\begin{pmatrix}
a^{-1}w_{n+2} & cw_{n+1}%
\end{pmatrix}%
.  \label{42}
\end{equation}%
The following result is a generalization of $\left( 2\right) $ and $\left(
4\right) $ of Lemma \ref{l1}.

\begin{theorem}
\label{t6}For positive integers $n$ and $m$, we have 
\begin{equation*}
\left(\frac{b}{a} \right)^{\zeta(mn+n)} w_{n+1} w_m + \left( \frac{b}{a}%
\right)^{\zeta(mn+m)} c w_n w_{m-1} = w_1 w_{m+n} + \left(\frac{b}{a}%
\right)^{\zeta(m+n)} c w_0 w_{m+n-1}.
\end{equation*}
\end{theorem}

\begin{proof}
For even $n$ and odd $m$, by (\ref{36}) and (\ref{42}), 
\begin{align*}
& (ab)^{\frac{n+m-1}{2}}\left( w_{n+1}w_{m}+\frac{b}{a}cw_{n}w_{m-1}\right)
=(ab)^{\frac{n+m-1}{2}}%
\begin{pmatrix}
w_{n+1} & bcw_{n}%
\end{pmatrix}%
\begin{pmatrix}
w_{m} \\ 
a^{-1}w_{m-1}%
\end{pmatrix}
\\
& =%
\begin{pmatrix}
w_{1} & bcw_{0}%
\end{pmatrix}%
A^{n}A^{m-1}%
\begin{pmatrix}
w_{1} \\ 
a^{-1}w_{0}%
\end{pmatrix}%
=%
\begin{pmatrix}
w_{1} & bcw_{0}%
\end{pmatrix}%
A^{m+n-1}%
\begin{pmatrix}
w_{1} \\ 
a^{-1}w_{0}%
\end{pmatrix}
\\
& =(ab)^{\frac{m+n-1}{2}}%
\begin{pmatrix}
w_{1} & bcw_{0}%
\end{pmatrix}%
\begin{pmatrix}
w_{m+n} \\ 
a^{-1}w_{m+n-1}%
\end{pmatrix}%
.
\end{align*}%
And hence the result follows. For the case of odd $n$ and even $m$, we do a
similar computation on 
\begin{equation*}
(ab)^{\frac{n+m-1}{2}}%
\begin{pmatrix}
a^{-1}w_{m} & cw_{m-1}%
\end{pmatrix}%
\begin{pmatrix}
bcw_{n+1} \\ 
cw_{n}%
\end{pmatrix}%
.
\end{equation*}%
For even $n$ and even $m$, we do a similar computation on 
\begin{equation*}
(ab)^{\frac{m+n-2}{2}}%
\begin{pmatrix}
a^{-1}w_{m} & cw_{m-1}%
\end{pmatrix}%
\begin{pmatrix}
w_{n+1} \\ 
a^{-1}w_{n}%
\end{pmatrix}%
.
\end{equation*}%
For odd $n$ and odd $m$, we do a similar computation on 
\begin{equation*}
(ab)^{\frac{m+n-2}{2}}%
\begin{pmatrix}
w_{m} & bcw_{m-1}%
\end{pmatrix}%
\begin{pmatrix}
bcw_{n+1} \\ 
cw_{n}%
\end{pmatrix}%
.
\end{equation*}
\end{proof}

By substituting $m=n+1$ in Theorem \ref{t6}, we get the following corollary,
which is a generalization of the classical result $F_{n+1}^2 +F_n^2
=F_{2n+1} $ for Fibonacci numbers.

\begin{corollary}
\label{c1}For positive integer $n$, we have 
\begin{equation*}
\left(\frac{b}{a} \right)^{\zeta(n)} w_{n+1}^2+ \left( \frac{b}{a}%
\right)^{\zeta(n+1)} c w_n^2 = w_1 w_{2n+1} + \left(\frac{b}{a}\right) c w_0
w_{2n}.
\end{equation*}
\end{corollary}

By Corollary \ref{c1}, we prove the following result, which is a
generalization of the classical result $F_{n+1}^2 - F_{n-1}^2 =F_{2n}$ for
Fibonacci numbers.

\begin{theorem}
\label{t7}For positive integer $n$, we have 
\begin{equation*}
w_{n+1}^2 -c^2 w_{n-1}^2 = a^{\zeta(n)} b^{\zeta(n+1)} (w_1 w_{2n} + c w_0
w_{2n-1} ).
\end{equation*}
\end{theorem}

\begin{proof}
For even $n$, by Corollary \ref{c1}, 
\begin{align*}
& w_{n+1}^{2}-c^{2}w_{n-1}^{2}=\left( w_{n+1}+\frac{b}{a}cw_{n}^{2}\right)
-\left( \frac{b}{a}cw_{n}^{2}+c^{2}w_{n-1}^{2}\right) \\
& =\left( w_{1}w_{2n+1}+\frac{b}{a}cw_{0}w_{2n}\right) -c\left(
w_{1}w_{2n-1}+\frac{b}{a}cw_{0}w_{2n-2}\right) \\
& =w_{1}\left( w_{2n+1}-cw_{2n-1}\right) +\frac{b}{a}cw_{0}\left(
w_{2n}-cw_{2n-2}\right) =bw_{1}w_{2n}+bcw_{0}w_{2n-1}.
\end{align*}%
For odd $n$, by Corollary \ref{c1}, 
\begin{align*}
& \left( \frac{b}{a}\right) ^{2}\left( w_{n+1}^{2}-c^{2}w_{n-1}^{2}\right) =%
\frac{b}{a}\left( \frac{b}{a}w_{n+1}^{2}+cw_{n}^{2}\right) -\frac{b}{a}%
c\left( w_{n}^{2}+\frac{b}{a}cw_{n-1}^{2}\right) \\
& =\frac{b}{a}\left( w_{1}w_{2n+1}+\frac{b}{a}cw_{0}w_{2n}\right) -\frac{b}{a%
}c\left( w_{1}w_{2n-1}+\frac{b}{a}cw_{0}w_{2n-2}\right) \\
& =\frac{b}{a}\left( w_{1}(w_{2n+1}-cw_{2n-1})+\frac{b}{a}%
cw_{0}(w_{2n}-cw_{2n-2})\right) =\frac{b}{a}\left(
bw_{1}w_{2n}+bcw_{0}w_{2n-1}\right) .
\end{align*}%
So the proof is complete after simple algebraic manipulations of both sides
of the equation.
\end{proof}

\textbf{Competing interests}

The authors declare that they have no competing interests.

\textbf{Authors' contributions}

All authors contributed equally to the manuscript and typed, read and
approved the final manuscript.

\textbf{Funding}

Not applicable.

\bigskip

{\small \noindent\textbf{Elif Tan} }

{\small \noindent Department of Mathematics}

{\small \noindent Ankara University }

{\small \noindent Ankara, Turkey }

{\small \noindent E-mail: etan@ankara.edu.tr}\newline

{\small \noindent\textbf{Ho-Hon Leung} }

{\small \noindent Department of Mathematical Sciences}

{\small \noindent United Arab Emirates University }

{\small \noindent Al Ain, 15551, United Arab Emirates }

{\small \noindent E-mail: hohon.leung@uaeu.ac.ae}\newline

\end{document}